\documentclass[11pt]{amsart}

\usepackage{amsmath}
\usepackage{amsthm}
\usepackage{amssymb}
\usepackage{mathrsfs}
\usepackage{verbatim}
\usepackage{esint}
\usepackage[foot]{amsaddr}

\usepackage{xcolor}
{%
	\setcounter{enumi}{0}

	\begin{enumerate}}%
	{\end{enumerate} }

{%
	\setcounter{enumi}{0}

	\begin{enumerate}}%
	{\end{enumerate} }

\usepackage[colorlinks=true,linkcolor=blue,citecolor=red]{hyperref}

\usepackage[nameinlink]{cleveref}

\newtheorem{theorem}{Theorem}[section]
\newtheorem{lemma}[theorem]{Lemma}
\newtheorem{proposition}[theorem]{Proposition}
\newtheorem{corollary}[theorem]{Corollary}

\theoremstyle{definition}
\newtheorem{definition}[theorem]{Definition}
\newtheorem{remark}[theorem]{Remark}
\newtheorem{example}[theorem]{Example}

\numberwithin{equation}{section}

\allowdisplaybreaks

\newcommand*\re{\mathbb{R}}
\newcommand*\rn{\mathbb{R}^n}
\newcommand*\sph{\mathbb{S}}

\begin{document}
\title[Fractional Sobolev--Poincar\'e and (localized) Hardy inequalities]{Characterization of fractional Sobolev--Poincar\'e and (localized) Hardy inequalities}
\author{Firoj Sk}
\address{Analysis And Partial Differential Equations Unit, Okinawa Institute Of Science And Technology, 1919-1 Tancha, Onna-Son, Okinawa 904-0495, Japan.}
\email{firojmaciitk7@gmail.com}

\smallskip

\maketitle
\begin{abstract}
In this paper, we prove capacitary versions of the fractional Sobolev--Poincar\'e inequalities. We characterize localized variant of the boundary fractional Sobolev--Poincar\'e inequalities through uniform fatness condition of the domain in $\rn$. Existence type results on the fractional Hardy inequality in the supercritical case $sp>n$ for $s\in(0,1)$, $p>1$ are established. 
\end{abstract}
\smallskip

\section{Introduction and main results}
The central aim of this paper is to study the Sobolev--Poincar\'e inequality, pointwise Hardy inequality and the Hardy inequality under some  assumptions on the domain in  the case of the fractional Sobolev spaces. Precise condition on the domain will be clarified later. It is well known that the classical Sobolev--Poincar\'e inequality states that for a bounded domain $\Omega\subset\rn$ with $C^1$ boundary and  $1\leq p<n$, there exists a constant $C=C(n,p)>0$ such that 
\begin{equation}\label{SPI}
  \left(\int_{\Omega}|u(x)-u_\Omega|^{p^*}dx\right)^{1/p^*}\leq C\left(\int_{\Omega}|\nabla u(x)|^{p}dx\right)^{1/p} \text{ for all }u\in W^{1,p}(\Omega), 
\end{equation}
where $p^*=\frac{np}{n-p}$ denotes the Sobolev critical exponent and the space $W^{1,p}(\Omega)$ is the usual classical Sobolev space, see for example \cite[Chapter~4]{EvGa} in the case of ball. A capacitary variant of the Sobolev--Poincar\'e inequality \cref{SPI} were considered in \cite{KilKos} and for weighted case, see \cite{Mi}. The well known classical (boundary) Hardy inequality states that for a bounded Lipschitz domain $\Omega\subset\rn$, $1\leq p<\infty$, there exists a constant $C=C(n,p,\Omega)>0$ such that for any $u\in C_c^\infty(\Omega)$
\begin{equation}\label{local HI}
    \int_{\Omega}\frac{|u(x)|^p}{\delta(x)^p}dx\leq C\int_{\Omega}|\nabla u(x)|^pdx,
\end{equation}
where $\delta(x):=\text{dist}(x,\partial\Omega)$. The existence of the Hardy inequality \cref{local HI} for every open set $\Omega\subset\rn$ when $p>n$ has been investigated independently by \cite{Lewis} and \cite{Wannebo}. Also observe that both references deal with the case $p\leq n$, as well, where the validity of \cref{local HI} has been established through the uniform fatness condition of the complement $\Omega^c$. One can obtain the classical Hardy inequality \cref{local HI} by applying appropriately the Hardy--Littlewood--Wiener maximal function theorem on a pointwise Hardy inequality, see \cite{Pi, KiLeVa, KiMa} where they have introduced pointwise Hardy inequality through a maximal operator. Necessary and sufficient conditions are provided for pointwise Hardy inequalities in \cite{KoLeTu} and see \cite{Le} for weighted case.
 
Let $\Omega\subseteq\rn$ be any open set, and let $0<s<1$, $1\leq p<\infty$, the fractional Sobolev space $W^{s,p}(\Omega)$ is defined as
$$	
W^{s,p}(\Omega):=\left\{u\in L^p(\Omega):[u]_{s,p,\Omega}<\infty\right\},
$$
endowed with the so-called fractional Sobolev norm, given by 
$$
||u||_{s,p,\Omega}:=\left(\|u\|_{L^p(\Omega)}^p+[u]_{s,p,\Omega}^p
\right)^\frac{1}{p},
$$
where
$$
[u]_{s,p,\Omega}^p:=
\int_{\Omega}\int_{\Omega}\frac{|u(x)-u(y)|^p}{|x-y|^{n+sp}}dxdy, 
$$
is the Gagliardo seminorm. For the study of fractional Sobolev spaces in a systematic way we refer to \cite{BrGoVa, BrSa, DiPaVa, DyKi1, FiSeVa} and references therein. At this stage, we consider two more Banach spaces $W^{s,p}_0(\Omega)$ and $W^{s,p}_{\Omega}(\rn)$ defined as the closure of the space $C_c^\infty(\Omega)$ with the norms $||\cdot||_{s,p,\Omega}$ and $||\cdot||_{s,p,\rn}$ respectively. These two spaces arise naturally in studying weak solutions of the Dirichlet problems involving regional fractional $p$-Laplacian and fractional $p$-Laplacian operators respectively, see \cite{BrLiPa, BrPaSq, Chen, Fall, SeVa} and references therein.  
If $\Omega$ is a bounded Lipschitz domain and $1<p<\infty$, then one has
$$
W^{s,p}_{\Omega}(\rn)=\{u\in W^{s,p}(\rn):u=0\text{ in }\Omega^c\},
$$
see \cite[Proposition~B.1]{BrPaSq}. Moreover, $W^{s,p}_{\Omega}(\rn)=W^{s,p}_0(\Omega)$ provided $sp\neq1$, see for instance \cite[Proposition~B.1]{BrLiPa}.

In the spirit of local case, we introduce what we call variational Sobolev capacity in fractional Sobolev spaces.
\begin{definition}
   Let $0<s<1$, $p\in[1,\infty)$ and $\Omega\subseteq\rn$ be an open set. For a compact set $K\subset\Omega$, variational $(s,p)$-Sobolev capacity is defined by 
   \begin{equation}\label{def: capacity}
       \text{Cap}_{s,p}(K,\Omega):=\inf\left\{[u]_{s,p,\Omega}^p: u\in C_c^\infty(\Omega),\; u\geq 1\text{ on }K\right\}.
   \end{equation}
 For an open set $A\subset\Omega$, variational $(s,p)$-Sobolev capacity is defined by
 $$
 \text{Cap}_{s,p}(A,\Omega)=\sup\left\{\text{Cap}_{s,p}(K,\Omega):K\subset A,\; K \text{ is compact }\right\},
 $$ and for an arbitrary set $E\subset\Omega$, variational $(s,p)$-Sobolev capacity is defined by
 $$
 \text{Cap}_{s,p}(E,\Omega)=\inf\left\{\text{Cap}_{s,p}(A,\Omega):E\subset A,\; A \text{ is open }\right\}.
 $$
 \noindent Using standard approximation argument, we can replace $C_c^\infty(\Omega)$ by a bigger space $W^{s,p}_0(\Omega)\cap C(\Omega)$ in the definition of capacity \cref{def: capacity}.
\end{definition} 
\begin{remark}\label{restriction of function in cap def}
It is worth mentioning that in the definition of capacity \eqref{def: capacity} one can restricts the function $u\in C_c^\infty(\Omega)$ such that $u=1$ in a neighbourhood $\mathcal{N}(K)\subset\Omega$ of $K$ and $0\leq u\leq1$ in $\Omega$, see for instance \cite[Theorem~2.1]{ShXi}.
\end{remark}
\begin{definition}
    Let $0<s<1$, $p\in[1,\infty)$. We say that a property holds $(s,p)$-quasi everywhere (in short $(s,p)$-q.e.) if it holds except for a set of capacity zero.\\
   \smallskip
   We say a function $u:\Omega\to\re$ is $(s,p)$-quasi continuous (in short $(s,p)$-q.c.) in $\Omega$ if for every $\epsilon>0$ there exists an open set $E\subset\Omega$ such that $\text{Cap}_{s,p}(E,\Omega)<\epsilon$ and $u|_{\Omega\setminus E}$ is continuous.
\end{definition}

\begin{remark}\label{level set open}
 We observe that, for any  $\lambda\in\re$,  the set $\{x\in\Omega: u(x)\neq\lambda\}\cup E$ is open in $\Omega$ and hence the set $\{x\in\Omega: u(x)=\lambda\}\cap E^c$ is closed in $\Omega$, although $\{x\in\Omega: u(x)\neq\lambda\}$ need not be open. Indeed, by definition of the $(s,p)$-quasi continuous, $u|_{\Omega\setminus E}$ is continuous. Thus, the set $\{x\in\Omega: u(x)\neq\lambda\}\setminus E$ is open in $\Omega\setminus E$ with respect to the relative topology. Therefore, there exists an open set $O$ in $\Omega$ such that $\{x\in\Omega: u(x)\neq\lambda\}\setminus E=O\setminus E$ and this implies  $\{x\in\Omega: u(x)\neq\lambda\}\cup E=O\cup E$ is open in $\Omega.$ In particular, from this observation we have $Z(u;E^c)=\{x\in\Omega: u(x)=0\}\cap E^c$ is a closed set in $\Omega.$
\end{remark}

\begin{remark}\label{qc represent}
It is important to note that from \cite[Theorem~2.2]{ShXi1}, for a compact set $K\subset\Omega$ we have the following characterization for $\text{Cap}_{s,p}(K,\Omega)$ via $(s,p)$-q.e. property
$$
\text{Cap}_{s,p}(K,\Omega)=\inf\left\{[u]_{s,p,\Omega}^p: u\in W^{s,p}_0(\Omega),\; u\geq 1\; (s,p)\text{-q.e. on }K\right\}.
$$

\end{remark}
\noindent In recent years, many researchers have shown their interest in studying variational Sobolev capacities, see \cite{Adams, AdHe, AdXi, HeKiMa} for the case of classical Sobolev spaces and \cite{BrSa, ShXi, ShXi1, ShZh, Warma} for the case of fractional Sobolev spaces.
\smallskip

Before outlining the main results in the present paper in a precise manner, we need to introduce some terminologies and definitions.
Let $0\leq\alpha<1$ and $R>0$. For a locally integrable function $u$, the fractional maximal function is defined by 
\begin{equation*}\label{fractional maximal function}
    M_{\alpha,R}(u)(x):=\sup\limits_{0<r<R}r^{\alpha}\fint_{B_r(x)}|u(y)|dy.
\end{equation*}
If $R=\infty$, then we shall simply write $ M_{\alpha,R}$ by $ M_{\alpha}$ and for $\alpha=0,\;R=\infty$, we have the usual maximal function. Let $0<\beta<\infty$, the fractional sharp maximal function of a locally integrable function $u$ is defined by
\begin{equation*}\label{fractional sharp maximal function}
 u^{\#}_{\beta,R}(x):=\sup\limits_{0<r<R}r^{-\beta}\fint_{B_r(x)}|u(y)-u_{B_r(x)}|dy.   
\end{equation*}
 If $R=\infty$, then we shall simply write $u^{\#}_{\beta,R}$ by $u^{\#}_{\beta}.$
\smallskip
\begin{definition}[\textbf{Pointwise fractional $p$-Hardy inequality}]
   Let $0<s<1$, $0\leq\alpha<1$ and $p\in[1,\infty)$. We say that an open set $\Omega\subsetneq\rn$ with non-empty boundary admits \textit{pointwise fractional $p$-Hardy's inequality} if there exist constants $C>0$ and $\sigma\geq 1$ such that
   \begin{equation}\label{def: ptwise frac HI}
       |u(x)|\leq C\delta(x)^{s-\frac{\alpha}{p}}\left(M_{\alpha,\;\sigma\delta(x)}(|D^s_p u|)^p(x)\right)^{1/p} \text{ for all }u\in C_c^\infty(\Omega),
   \end{equation}
   and for all $x\in\Omega$, where $|D^s_p u|(x):=\left(\displaystyle\int_{\rn}\frac{ |u(x)-u(y)|^p}{|x-y|^{n+sp}}dy\right)^{\frac{1}{p}}.$
 \end{definition}
 \begin{definition}[\textbf{Uniformly $(s,p)$-fat set}]
   Let $0<s<1$ and $1< p<\infty$ with $sp>1.$ We say that a closed set $E\subset\rn$ is \textit{uniformly $(s,p)$-fat set} if there exists a constant $\gamma>0$ such that $$\text{Cap}_{s,p}(E\cap\overline{B_r(x)},2B_r(x))\geq \gamma\;\text{Cap}_{s,p}(\overline{B_r(x)},2B_r(x)), \text{ for all } x\in E \text{ and } r>0.$$  
\end{definition}
 Let us now describe our results in this paper before formulating these. \Cref{FSPI} is a capacitary version of the fractional Sobolev--Poincar\'e inequality, which is motivated by the result of \cite{KilKos}, and whereas, \cref{Boundary FSPI} gives a characterization of uniformly $(s,p)$-fat set through a boundary fractional Sobolev--Poincar\'e type inequality. As an application of \cref{Boundary FSPI}, at the end of \cref{Section3}, we provide various classes of domains that are uniformly $(s,p)$-fat set. The existence issue regarding  fractional $p$-Hardy's inequality \cref{hardy inq} in the supercritical case $sp>n$ for any proper open set is addressed in \cref{Fractional Hardy Inequality}.
 This result can be obtained by proving an appropriate pointwise fractional Hardy type inequality and applying the maximal function theorem.

Our main results are stated below.

\begin{theorem}
\label{FSPI}
 Let $0<s<1, \,p\in(1,\infty)$ with $sp>1$ and suppose $u\in W^{s,p}(B)$ be a $(s,p)$-quasi continuous function, where $B=B_r(x_0)\subset\rn$ is an open ball of radius $r>0.$ Let $1\leq q\leq p^*_s$ for $sp<n$ and $1\leq q<\infty$ for $sp\geq n.$ Then there exists a constant $C=C(n,s,p,q)>0$ such that 
    \begin{equation*}
        \left(\fint_{B}|u(x)|^{q}dx\right)^{\frac{1}{q}}\leq C\left(\frac{ 1}{\text{Cap}_{s,p}\left(Z(u;E^c)\cap\frac{1}{2}\overline{ B},B\right)}\int_{B}\int_{B}\frac{|u(x)-u(y)|^p}{|x-y|^{n+sp}}dxdy\right)^{\frac{1}{p}},
    \end{equation*} 
    where the closed set $Z(u;E^c)$ as in \cref{level set open}. 
\end{theorem}
\begin{theorem}
\label{Boundary FSPI}
    Let $0<s<1$, $p\in(1,\infty)$ with $sp>1$ and let $\Omega$ be any proper open set in $\rn$. Let $1\leq q\leq p^*_s$ for $sp<n$ and $1\leq q<\infty$ for $sp\geq n.$ Then $\rn\setminus\Omega$ is uniformly $(s,p)$-fat set with a constant $\gamma$ if and only if for any $z\in\rn\setminus\Omega$, $r>0$
    \begin{equation}\label{fat iff FSPI}
   \left(\fint_{B_r(z)}|u(x)|^{q}dx\right)^{\frac{1}{q}}\leq C\,\gamma^{-\frac{1}{p}}\, r^{s-\frac{n}{p}}\left(\int_{B_r(z)}\int_{B_r(z)}\frac{|u(x)-u(y)|^p}{|x-y|^{n+sp}}dxdy\right)^{\frac{1}{p}}\;
    \end{equation}
     for all $u\in C_c^\infty(\Omega)$, and where $C=C(n,s,p,q)$ is a constant. 
\end{theorem}

\begin{theorem}
\label{Fractional Hardy Inequality}
Let $\Omega$ be any open set in $\rn$ with $\Omega\neq\rn$. Let $0<s<1$ and $p>1$ such that $sp>n$. Then there exists a constant $C=C(n,s,p)>0$ such that the fractional Hardy inequality holds that is
\begin{equation*}\label{hardy inq}
\int_{\Omega}\frac{ |u(x)|^p}{\delta(x)^{sp}}dx\leq C\int_{\rn}\int_{\rn}\frac{ |u(x)-u(y)|^p}{|x-y|^{n+sp}}dxdy,\;\text{ for all }u\in W^{s,p}_{\Omega}(\rn).
\end{equation*} Furthermore, the regional fractional Hardy inequality holds that is
\begin{equation*}
\int_{\Omega}\frac{ |u(x)|^p}{\delta(x)^{sp}}dx\leq C\int_{\Omega}\int_{\Omega}\frac{ |u(x)-u(y)|^p}{|x-y|^{n+sp}}dxdy,\;\text{ for all }u\in W^{s,p}_{0}(\Omega).
\end{equation*}
\end{theorem}
Recently, in \cite{DyLeVa} the authors studied capacitary versions of fractional Poincar\'e, pointwise, and localized fractional Hardy inequalities in a metric measure space. However, their results involve the Assouad codimension of the domain, and certain restrictions on  functions. The study of Hardy inequalities in fractional Sobolev spaces has emerged as an intriguing research area in recent times. There is numerous literature available on this topic. For details discussion on the sharp constants in fractional Hardy inequalities, we refer to \cite{BiBrZa, BoDy, DyKi,  LoSl} and references therein.

This paper organized in the following way: In \cref{Section2} we collect some known results and discussed some necessary preliminaries. Proofs of \cref{FSPI,Boundary FSPI,Fractional Hardy Inequality} along with some further results are given in \cref{Section3}.

\section{Preliminaries and Known results}\label{Section2}
Throughout the paper we shall assume the following notations, unless mentioned otherwise explicitly:
\begin{itemize}
    \item $\Omega$ is an open connected set in $\rn$, $0<s<1$, $1\leq p<\infty,\;n\in\mathbb{N}.$
    \smallskip 
    \item $p^*_s=\frac{np}{n-sp}$ is the fractional Sobolev critical exponent for $ sp<n$. 
    \smallskip
    \item $\overline{\Omega}$ is the closure of $\Omega.$
    \smallskip
    \item $|\Omega|$ is the Lebesgue measure of $\Omega.$
    \smallskip
    \item $u_{\Omega}=\fint_{\Omega}udx=\frac{1}{|\Omega|}\int_{\Omega}udx$ is the average of the function $u$ in $\Omega.$
    \smallskip
    \item $X^c$ is the complement of the set $X$ in the appropriate ambient space.
    \smallskip
    \item $B_r(x)$ is an open ball centered at $x$ of radius $r>0$.
    \smallskip
    \item $\sph^{k-1}$ is the unit sphere in $\re^{k}.$
    \smallskip
    \item $c,\;C,\;C(*,*,\cdots, *)>0$ denote generic constants that will appear in the estimate and need not be the same as in the preceding steps; the value depends on the quantities indicated by  $*$'s. 
\end{itemize}
We start with some known results and some technical lemmas that will be required to prove our results. 
\begin{lemma}\label{cap properties}
   Let $s\in(0,1)$, $1< p<\infty$ with $sp>1$. Then the following properties of capacity hold: 
   \begin{enumerate}
       \item[$a)$]\textbf{(Ball estimate:)}  $\text{Cap}_{s,p}\left(\overline{B_r(x)},2B_r(x)\right)=C(n,s,p)\; r^{n-sp}$, for a constant $C(n,s,p)>0.$
       \smallskip
       \item[$b)$]\textbf{(Monotonicity:)}
            If $K_1\subseteq K_2\subset\Omega$, where $K_i$'s are compact sets, one has
       $$
       \text{Cap}_{s,p}(K_1,\Omega)\leq\text{Cap}_{s,p}(K_2,\Omega).
       $$
   \end{enumerate}
\end{lemma}
\begin{proof}
     $a)$ It follows from \cite[Theorem~2.2]{ShXi} by choosing the radius of the ball appropriately.\\
     \smallskip
     $b)$ It is an immediate consequence of the definition of $\text{Cap}_{s,p}(\cdot,\Omega)$. 
\end{proof}
The proof of the following lemma can be found in \cite{LiWa}, however we include the proof of it for the sake of completeness. 
\begin{lemma}
\label{seminorm equivalent}
   Let $\Omega\subsetneq\rn$ be an open set, $s\in(0,1)$ and $p\in(0,\infty)$ such that $sp>1.$ Then 
  $$
  \int_{\rn}\int_{\rn}\frac{|u(x)-u(y)|^p}{|x-y|^{n+sp}}dxdy
  \leq C\int_{\Omega}\int_{\Omega}\frac{|u(x)-u(y)|^p}{|x-y|^{n+sp}}dxdy\;\text{ for all }u\in C^\infty_c(\Omega),
  $$
  where $C=C(n,s,p)$ is a positive constant and does not depend on the domain $\Omega$.
\end{lemma}
\begin{proof}
    Let $u\in C_c^\infty(\Omega)$. Then we have
    \begin{multline*}
        \int_{\rn}\int_{\rn}\frac{|u(x)-u(y)|^p}{|x-y|^{n+sp}}dxdy\\
        =\int_{\Omega}\int_{\Omega}\frac{|u(x)-u(y)|^p}{|x-y|^{n+sp}}dxdy+2\int_{\Omega}|u(x)|^p\left(\int_{\Omega^c}\frac{dy}{|x-y|^{n+sp}}\right)dx.
\end{multline*}
Note that, for $x\in\Omega$, we have
\begin{equation*}
    \int_{\Omega^c}\frac{dy}{|x-y|^{n+sp}}\leq\int_{\sph^{n-1}}dw\int_{d_{w,\Omega}(x)}^\infty\frac{dr}{r^{sp+1}}=\frac{1}{sp}\int_{\sph^{n-1}}\frac{dw}{d_{w,\Omega}(x)^{sp}}=\frac{C(n,s,p)}{m_{sp}(x)^{sp}},
\end{equation*}
where $d_{w,\Omega}(x)=\inf\{|t|:x+tw\in\Omega^c\}$ (see, \cite{LoSl}) and
$$
m_{sp}(x)^{sp}=\left.\frac{2\pi^{\frac{n-1}{2}}\Gamma\left(\frac{1+sp}{2}\right)}{\Gamma\left(\frac{n+sp}{2}\right)}\middle/\int_{\sph^{n-1}}\frac{dw}{d_{w,\Omega}(x)^{sp}}\right. .
$$
Thus, we obtain
\begin{multline*}
            \int_{\rn}\int_{\rn}\frac{|u(x)-u(y)|^p}{|x-y|^{n+sp}}dxdy\leq\int_{\Omega}\int_{\Omega}\frac{|u(x)-u(y)|^p}{|x-y|^{n+sp}}dxdy+2C(n,s,p)\int_{\Omega}\frac{|u(x)|^p}{m_{sp}(x)^{sp}}dx\\
            \leq C(n,s,p)\int_{\Omega}\int_{\Omega}\frac{|u(x)-u(y)|^p}{|x-y|^{n+sp}}dxdy,
\end{multline*}
where in the last step we have used the fractional Hardy inequality of Loss--Sloane \cite[Theorem~1.2]{LoSl}. Hence the lemma follows.
\end{proof}

\begin{lemma}\label{telescoping lemma}
Suppose $u\in L^1_{\text{loc}}(\rn)$ and let $0<\beta<\infty$. Then there is a constant $C=C(n,\beta)$ and a set $A\subset\rn$ with $|A|=0$ such that 
\begin{equation*}
    |u(x)-u(y)|\leq C|x-y|^{\beta}\left(u^{\#}_{\beta,4|x-y|}(x)+u^{\#}_{\beta,4|x-y|}(y)\right),\;\text{ for all }x,y\in\rn\setminus A.
\end{equation*}
\end{lemma}	
\begin{proof}
Let $S_u$ be the set of all Lebesgue points of the function $u$ and set $A=S_u^c$. Then, by Lebesgue differentiation theorem we have $|A|=0$. Now fix $x\in S_u$, $0<r<\infty$ and we denote $B_k=B_{\frac{ r}{2^k}}(x)$, $k=0,1,2,\cdots$
\begin{multline*}
|u(x)-u_{B_{r}(x)}|
=\left|\lim\limits_{m\to\infty}\sum_{k=0}^{m}\left(u_{B_{k+1}}-u_{B_k}\right)\right|
\leq\sum_{k=0}^{\infty}\left|u_{B_{k+1}}-u_{B_k}\right|\\
=\sum_{k=0}^{\infty}\left|\frac{ 1}{|B_{k+1}|}\int_{B_{k+1}}u(z)dz-u_{B_k}\right|
\leq\sum_{k=0}^{\infty}\frac{ 1}{|B_{k+1}|}\int_{B_{k+1}}\left|u(z)-u_{B_k}\right|dz\\
\leq\sum_{k=0}^{\infty}\frac{ |B_k|}{|B_{k+1}|}\fint_{B_{k}}\left|u(z)-u_{B_k}\right|dz
=C(n)\sum_{k=0}^{\infty} \left(\frac{ r}{2^k}\right)^{\beta}\left(\frac{ r}{2^k}\right)^{-\beta}\fint_{B_{k}}\left|u(z)-u_{B_k}\right|dz\\
\leq C(n,\beta)r^{\beta}u^{\#}_{\beta,r}(x).
\end{multline*}
Let $y\in B_{r}(x)\setminus A$ and we have $B_r(x)\subset B_{2r}(y)$. Now 
\begin{multline*}
|u(y)-u_{B_{r}(x)}|\leq |u(y)-u_{B_{2r}(y)}|+|u_{B_{2r}(y)}-u_{B_{r}(x)}|\\
\leq C_1\;r^{\beta}u^{\#}_{\beta,2r}(y)+\fint_{B_r(x)}|u(z)-u_{B_{2r}(y)}|dz\\
\leq C_1r^{\beta}u^{\#}_{\beta,2r}(y)+C_2(n)\fint_{B_{2r}(y)}|u(z)-u_{B_{2r}(y)}|dz
\leq C(n,\beta)r^{\beta}u^{\#}_{\beta,2r}(y).
\end{multline*}
Let $x,y\in\rn\setminus A$ with $x\neq y$ and let $r=2|x-y|$. Then $x,y\in B_r(x)$ and hence we have 
\begin{equation*}
|u(x)-u(y)|\leq |u(x)-u_{B_{r}(x)}|+|u(y)-u_{B_{r}(x)}|\leq C(n,\beta)r^{\beta}\left(u^{\#}_{\beta,2r}(x)+u^{\#}_{\beta,2r}(y)\right).
\end{equation*}
This completes the proof of the lemma.
\end{proof}	
	
\begin{lemma}\label{reltn btw sharp and frac maximal}
    Let $0\leq\alpha<s<1,\;R>0$. Suppose $u\in W^{s,p}(\rn)$. Then for any $x\in\rn$ there is a constant $C=C(n,s,p)$ such that
\begin{equation*}
    u^{\#}_{s-\alpha,R}(x)\leq C M_{\alpha,R}\left(|D^s_p u|\right)(x),
\end{equation*}
where $|D^s_p u|(x)=\left(\displaystyle\int_{\rn}\frac{ |u(x)-u(y)|^p}{|x-y|^{n+sp}}dy\right)^{1/p}$.
\end{lemma}	
\begin{proof}
Using H\"older inequality we have 
\begin{multline*}
|u((y)-u_{B_r(x)}|=\left|u(y)-\frac{ 1}{|B_{r}(x)|}\int_{ B_{r}(x)}u(z)dz\right|\leq\frac{ 1}{|B_{r}(x)|}\int_{ B_{r}(x)}|u(y)-u(z)|dz\\
\leq\frac{ 1}{|B_{r}(x)|^{1/p}}\left(\int_{ B_{r}(x)}\frac{ |u(y)-u(z)|^p}{|y-z|^{n+sp}}|y-z|^{n+sp}dz\right)^{1/p}.
\end{multline*}
Integrating over the ball $B_r(x)$ we obtain
\begin{multline*}
\fint_{ B_{r}(x)}|u((y)-u_{B_r(x)}|dy\leq\frac{C r^{\frac{n+sp}{p}}}{|B_r(x)|^{1/p}}\fint_{ B_{r}(x)}\left(\int_{ B_{r}(x)}\frac{ |u(y)-u(z)|^p}{|y-z|^{n+sp}}dz\right)^{1/p}dy\\
\leq\frac{C r^{\frac{n+sp}{p}}}{|B_r(x)|^{1/p}}\fint_{ B_{r}(x)}|D^s_p u|(y)dy\leq C r^{s-\alpha}r^{\alpha}\fint_{ B_{r}(x)}|D^s_p u|(y)dy,
\end{multline*}
where $C=C(n,s,p)$. Thus we get 
\begin{equation*}
  r^{\alpha-s}\fint_{ B_{r}(x)}|u((y)-u_{B_r(x)}|dy\leq C r^{\alpha}\fint_{ B_{r}(x)}|D^s_p u|(y)dy,  
\end{equation*}
 and consequently we have 
 $$u^{\#}_{s-\alpha,R}(x)\leq C M_{\alpha,R}\left(|D^s_p u|\right)(x),\text{ for any }x\in\rn\;\text{ and }R>0.$$
\end{proof}	
By the above two lemmas we immediately get the following result.
\begin{corollary}\label{useful result}
Let $u\in W^{s,p}(\rn)$ and $0\leq\alpha<s<1$. Then, there is a set $A\subset\rn$ with $|A|=0$ such that for all $x,y\in\rn\setminus A$ we have
\begin{equation}\label{maximal function inq}
  |u(x)-u(y)|\leq C|x-y|^{s-\alpha}\left(M_{\alpha,4|x-y|}(|D^s_p u|)(x)+M_{\alpha,4|x-y|}(|D^s_p u|)(y)\right),  
\end{equation}
where $C=C(n,s,\alpha,p)$ is a positive constant.
\end{corollary}
\begin{remark}
    In view of the above corollary, we say that \cref{maximal function inq} holds for almost every $x$ and $y.$ 
\end{remark}
\begin{proposition}\label{FPW q p}
Let $\Omega$ be an open set in $\rn$ and $0<s<1,\,1\leq p<\infty$. Assume that $u\in W^{s,p}(\Omega)$. Let $1\leq q\leq p^*_s$ for $sp<n$ and $1\leq q<\infty$ for $sp\geq n.$ Then there exists a constant $C=C(n,s,p,q)$ such that 
\begin{equation*}
\left(\fint_{B}|u(x)-u_{B}|^{q}dx\right)^{\frac{1}{q}} \leq C \left(r^{sp-n}\int_{B}\int_{B}\frac{|u(x)-u(y)|^p}{|x-y|^{n+sp}}dxdy\right)^{\frac{1}{p}}
\end{equation*}
for each ball $B=B_r(x_0)\subset\Omega.$
\end{proposition}
\begin{proof}
  Let $sp<n$. Then, applying improved fractional Sobolev--Poincar\'e inequality (see, \cite{HuVa, ShXi1}) for $B_{r}(x_0)$ we have 
  \begin{equation}\label{avg Sobolev-Poincare}
        \left(\fint_{B}|u(x)-u_B|^{p^*_s}dx\right)^{\frac{1}{p^*_s}} \leq C \left(r^{sp-n}\int_{B}\int_{B}\frac{|u(x)-u(y)|^p}{|x-y|^{n+sp}}dxdy\right)^{\frac{1}{p}}.
    \end{equation}
Now by H\"older inequality together with \cref{avg Sobolev-Poincare} we obtain the desired inequality in this case.

Let $sp\geq n.$ We choose $0<s^{\prime}<s$ such that $s^{\prime}p<n.$ Therefore, we have $u\in W^{s^{\prime},p}(\Omega)$ by \cite[Proposition 2.1]{DiPaVa}. If $q>p$ satisfying $q=\frac{np}{n-s^{\prime}p}$. Then by \cref{avg Sobolev-Poincare} we get
\begin{equation}\label{avg Sobolev q Poin }
        \left(\fint_{B}|u(x)-u_B|^{q}dx\right)^{\frac{1}{q}} \leq C r^{s^{\prime}-\frac{n}{p}} \left(\int_{B}\int_{B}\frac{|u(x)-u(y)|^p}{|x-y|^{n+s^{\prime}p}}dxdy\right)^{\frac{1}{p}}.
    \end{equation}
Now, for $x,y\in B$ we have $|x-y|<2r$. Since $s^{\prime}<s$, therefore we get 
$$
\left(\frac{|x-y|}{2r}\right)^{n+sp}<\left(\frac{|x-y|}{2r}\right)^{n+s^{\prime}p}.
$$
Using this to \cref{avg Sobolev q Poin }, we obtain 
\begin{equation*}
        \left(\fint_{B}|u(x)-u_B|^{q}dx\right)^{\frac{1}{q}} \leq C(n,s,p,q) r^{s-\frac{n}{p}} \left(\int_{B}\int_{B}\frac{|u(x)-u(y)|^p}{|x-y|^{n+sp}}dxdy\right)^{\frac{1}{p}}.
    \end{equation*}
Finally, for $q\leq p$, the result follows by using H\"older inequality and the above case with some $\ell>p\geq q$ on the left-hand side.
\end{proof}
\begin{proposition}[Fractional $(q,p)$-Poincar\'e inequality]\label{FP q p}
Let $0<s<1,\,1< p<\infty$ with $sp>1$. Let $1\leq q\leq p^*_s$ for $sp<n$ and $1\leq q<\infty$ for $sp\geq n.$ Then there exists a constant $C=C(n,s,p,q)$ such that 
\begin{equation*}
\left(\fint_{B}|u(x)|^{q}dx\right)^{\frac{1}{q}} \leq C \left(r^{sp-n}\int_{B}\int_{B}\frac{|u(x)-u(y)|^p}{|x-y|^{n+sp}}dxdy\right)^{\frac{1}{p}}
\end{equation*}
for each ball $B=B_r(x_0)\subset\rn$ and each $u\in W^{s,p}_{0}(B)$.
\end{proposition}
\begin{proof}
Since the claim follows from H\"older inequality for $q=1$, we thus assume that $q>1.$ Let $u\in C_c^\infty(B)$. Then $u=0$ in $2B\setminus B.$ Again, by H\"older inequality we have
\begin{equation}\label{avg est in 2B}
|u_{2B}|\leq\fint_{2B}|u(x)|dx=\fint_{2B}|u(x)|\chi_{B}(x)dx\leq2^{\frac{n}{q}-n}\left(\fint_{2B}|u(x)|^q dx\right)^{1/q}. 
\end{equation}
Now, using Minkowski inequality and the \cref{FPW q p} for $2B$ with the estimate \cref{avg est in 2B}, we obtain
\begin{multline*}
\left(\fint_{2B}|u(x)|^qdx\right)^{1/q}\leq\left(\fint_{2B}|u(x)-u_{2B}|^qdx\right)^{1/q}+|u_{2B}|
\\
\leq C(n,s,p,q)r^{s-\frac{n}{p}}[u]_{s,p,2B}+2^{\frac{n}{q}-n}\left(\fint_{2B}|u(x)|^q dx\right)^{1/q}. 
\end{multline*}
Since $2^{\frac{n}{q}-n}<1$ and using (Proposition 3.1, \cite{BiBr}) to estimate the seminorm in the above, we obtain
\begin{equation*}
  \left(\fint_{2B}|u(x)|^qdx\right)^{1/q}\leq Cr^{s-\frac{n}{p}}\left(\int_{B}\int_{B}\frac{|u(x)-u(y)|^p}{|x-y|^{n+sp}}dxdy\right)^{\frac{1}{p}}  
\end{equation*}
Finally, the average integral in the left hand side of the above can be taken on $B$ since $u=0$ in $2B\setminus B$. Hence the result follows by density. 
\end{proof}

\section{On Sobolev-Poincar\'e and localized Hardy inequalities}\label{Section3}
This section is devoted to proofs of \cref{FSPI,Boundary FSPI,Fractional Hardy Inequality}. We also address some further results in this context. 
\smallskip

\begin{proof}[\textbf{\protect{Proof of \cref{FSPI}}}.]
    If $u_B=\fint_{B}u=0$. Then, by \cref{FPW q p} we have \begin{equation}\label{th 19 FPI q p}
        \left(\fint_{B}|u(x)|^{q}dx\right)^{\frac{1}{q}} \leq C \left(r^{sp-n}\int_{B}\int_{B}\frac{|u(x)-u(y)|^p}{|x-y|^{n+sp}}dxdy\right)^{\frac{1}{p}}.
    \end{equation}
Using the capacity estimate \cref{cap properties} we get
\begin{equation}\label{cap est zero set}
  \text{Cap}_{s,p}\left(Z(u;E^c)\cap\frac{1}{2}\overline{B},B\right)\leq\text{Cap}_{s,p}\left(\frac{1}{2}\overline{B},B\right)=C(n,p,s)r^{n-sp}.  
\end{equation}
By exploiting this estimate in \cref{th 19 FPI q p} to get the desired result in this case.
\smallskip

\noindent If $u_B\neq 0$, without loss of generality we can take $u_B=1.$ Choose $\phi\in C_c^\infty(B)$ such that $\phi=1\text{ in }\frac{1}{2}\overline{B}$, $|\nabla\phi|\leq c/r$ and $0\leq\phi\leq 1.$ We define $\psi=\phi(u_B-u)$. Clearly, $\psi\in W^{s,p}_0(B)$ be a $(s,p)$-quasi continuous function and $\psi=1\text{ in }Z(u;E^c)\cap\frac{1}{2}\overline{B}.$ Therefore, by \cref{qc represent} we obtain
\begin{multline}\label{capacity upper bound}
 \text{Cap}_{s,p}\left(Z(u;E^c)\cap\frac{1}{2}\overline{B},B\right)
 \leq\int_{B}\int_{B}\frac{|\psi(x)-\psi(y)|^p}{|x-y|^{n+sp}}dxdy\\
 \leq 2^{p-1}\left(\int_{B}\int_{B}\frac{|\phi(x)|^p|u(x)-u(y)|^p}{|x-y|^{n+sp}}dxdy+\int_{B}\int_{B}\frac{|u_B-u(y)|^p|\phi(x)-\phi(y)|^p}{|x-y|^{n+sp}}dxdy\right)\\
 \leq 2^{p-1}\int_{B}\int_{B}\frac{|u(x)-u(y)|^p}{|x-y|^{n+sp}}dxdy+2^{p-1}\int_{B}|u(y)-u_B|^p\int_{B\cap\{x:|x-y|\leq\eta\}}\frac{|\phi(x)-\phi(y)|^p}{|x-y|^{n+sp}}dxdy\\
    +2^{p-1}\int_{B}|u(y)-u_B|^p\int_{B\cap\{x:|x-y|>\eta\}}\frac{|\phi(x)-\phi(y)|^p}{|x-y|^{n+sp}}dxdy\\
    =:2^{p-1}\int_{B}\int_{B}\frac{|u(x)-u(y)|^p}{|x-y|^{n+sp}}dxdy+2^{p-1}(I_1+I_2),
\end{multline}
where $\eta>0$ will be chosen later,
$$
I_1:=\int_{B}|u(y)-u_B|^p\int_{B\cap\{x:|x-y|\leq\eta\}}\frac{|\phi(x)-\phi(y)|^p}{|x-y|^{n+sp}}dxdy,
$$
and
$$
I_2:=\int_{B}|u(y)-u_B|^p\int_{B\cap\{x:|x-y|>\eta\}}\frac{|\phi(x)-\phi(y)|^p}{|x-y|^{n+sp}}dxdy.
$$
We estimate the above integrals $I_1,\;I_2$ in the following:
\smallskip

\noindent\textit{Estimate for $I_1$ :} Using the properties of the function $\phi$, we have
\begin{multline*}
    I_1=\int_{B}|u(y)-u_B|^p\int_{B\cap\{x:|x-y|\leq\eta\}}\frac{|\phi(x)-\phi(y)|^p}{|x-y|^{n+sp}}dxdy\\
    \leq\frac{c^p}{r^p}\int_{B}|u(y)-u_B|^p\int_{B\cap\{x:|x-y|\leq\eta\}}\frac{dx}{|x-y|^{n+sp-p}}dy\\
    \leq c^p r^{-p}\int_{B}|u(y)-u_B|^p\int_{|x-y|\leq\eta}\frac{dx}{|x-y|^{n+sp-p}}dy\leq\frac{C\eta^{p-sp}}{r^p}\int_{B}|u(y)-u_B|^pdy.
\end{multline*}
\smallskip

\noindent\textit{Estimate for $I_2$ :}
\begin{multline*}
    I_2=\int_{B}|u(y)-u_B|^p\int_{B\cap\{x:|x-y|>\eta\}}\frac{|\phi(x)-\phi(y)|^p}{|x-y|^{n+sp}}dxdy\\
    \leq 2^p\int_{B}|u(y)-u_B|^p\int_{B\cap\{x:|x-y|>\eta\}}\frac{dx}{|x-y|^{n+sp}}dy\\
    \leq 2^p\int_{B}|u(y)-u_B|^p\int_{|x-y|>\eta}\frac{dx}{|x-y|^{n+sp}}dy\leq C\eta^{-sp}\int_{B}|u(y)-u_B|^pdy.
\end{multline*}
Plugging the above two estimates into \cref{capacity upper bound}, we obtain
\begin{multline*}
    \text{Cap}_{s,p}\left(Z(u;E^c)\cap\frac{1}{2}\overline{B},B\right)\\
    \leq 2^{p-1}\int_{B}\int_{B}\frac{|u(x)-u(y)|^p}{|x-y|^{n+sp}}dxdy+2^{p-1}C\left(\frac{\eta^{p-sp}}{r^p}+\eta^{-sp}\right)\int_{B}|u(y)-u_B|^pdy.
\end{multline*}
Now, choose $\eta=r$ the radius of the ball $B$ and by \cref{FPW q p} for $q=p$, we have 
\begin{equation*}
    \text{Cap}_{s,p}\left(Z(u;E^c)\cap\frac{1}{2}\overline{B},B\right)\\
    \leq C\int_{B}\int_{B}\frac{|u(x)-u(y)|^p}{|x-y|^{n+sp}}dxdy.
\end{equation*}
Therefore, by above estimate we get 
\begin{equation}\label{avg esitmate via capacity}
    u_B=1\leq C\left(\text{Cap}_{s,p}\left(Z(u;E^c)\cap\frac{1}{2}\overline{B},B\right)^{-1}
    \int_{B}\int_{B}\frac{|u(x)-u(y)|^p}{|x-y|^{n+sp}}dxdy\right)^{\frac{1}{p}}.
\end{equation}
Now using Minkowski inequality, the estimate \cref{avg esitmate via capacity} together with \cref{FPW q p}, we obtain
\begin{align*}
    \left(\fint_{B}|u(x)|^{q}dx\right)^{\frac{1}{q}}
    &\leq\left(\fint_{B}|u(x)-u_B|^{q}dx\right)^{\frac{1}{q}}+u_B\\
    &\leq C\; r^{s-\frac{n}{p}}\left(\int_{B}\int_{B}\frac{|u(x)-u(y)|^p}{|x-y|^{n+sp}}dxdy\right)^{\frac{1}{p}}
    \\
    &+C\left(\text{Cap}_{s,p}\left(Z(u;E^c)\cap\frac{1}{2}\overline{B},B\right)^{-1}\int_{B}\int_{B}\frac{|u(x)-u(y)|^p}{|x-y|^{n+sp}}dxdy\right)^{\frac{1}{p}}\\
    &\leq C \left(\frac{ 1}{\text{Cap}_{s,p}\left(Z(u;E^c)\cap\frac{1}{2}\overline{ B},B\right)}\int_{B}\int_{B}\frac{|u(x)-u(y)|^p}{|x-y|^{n+sp}}dxdy\right)^{\frac{1}{p}}.
\end{align*}
In the last estimate we used \cref{cap est zero set}. This completes the proof of the theorem.
\end{proof}
\begin{remark}\label{th 19 for continous function}
 In view of \cref{FSPI}, if we assume that $u\in W^{s,p}(B)$ be a continuous function then we may replace $\text{Cap}_{s,p}\left(Z(u;E^c)\cap\frac{1}{2}\overline{ B},B\right)$ by $\text{Cap}_{s,p}\left(Z(u)\cap\frac{1}{2}\overline{ B},B\right)$, $Z(u)$ is the zero set of $u$.
\end{remark}
\begin{remark}
\Cref{FSPI} is essentially the best possible for the case $q=p^*_s$ in the following way: Let $1<sp<n$ and $F\subset\rn$ closed set. Take $x_0\in F$ be any and consider the ball $B_r(x_0)$ with $r>0$ small enough. Suppose there exists $C_F>0$ such that 
$$
\left(\int_{B_r(x_0)}|u(x)|^{p^*_s}dx\right)^{1/p^*_s}\leq C_F\left(\int_{B_r(x_0)}\int_{B_r(x_0)}\frac{|u(x)-u(y)|^p}{|x-y|^{n+sp}}dxdy\right)^{1/p},
$$
where $u\in W^{s,p}(B_{r}(x_0))$ is a $(s,p)$-quasi continuous function with $u=0$ in $F\cap\frac{1}{2}\overline{B_{r}(x_0)}$. Let $v\in W^{s,p}_0(\frac{3}{4}B_r(x_0))$ be a $(s,p)$-quasi continuous function with $v=1$ in $F\cap\frac{1}{2}\overline{B_{r}(x_0)}$. Define $w:=1-v\in W^{s,p}(B_{r}(x_0))$. Then $w$ is a $(s,p)$-quasi continuous function and $w=0$ in $F\cap\frac{1}{2}\overline{B_{r}(x_0)}$. Hence,
\begin{multline*}
    C_F\left(\int_{B_r(x_0)}\int_{B_r(x_0)}\frac{|v(x)-v(y)|^p}{|x-y|^{n+sp}}dxdy\right)^{\frac{1}{p}}\\
    =C_F\left(\int_{B_r(x_0)}\int_{B_r(x_0)}\frac{|w(x)-w(y)|^p}{|x-y|^{n+sp}}dxdy\right)^{\frac{1}{p}}\\
    \geq\left(\int_{B_r(x_0)}|w|^{p^*_s}dx\right)^{\frac{1}{p^*_s}}\geq\left|B_r(x_0)\setminus\frac{3}{4}B_r(x_0)\right|^{\frac{1}{p^*_s}}\geq c|B_r(x_0)|^{\frac{1}{p^*_s}}.
\end{multline*}
Therefore, taking infimum over all such $v$ we infer that
$$
C_F\;\text{Cap}_{s,p}\left(F\cap\frac{1}{2}\overline{B_{r}(x_0)}, B_{r}(x_0)\right)^{1/p}\geq c|B_r(x_0)|^{\frac{1}{p^*_s}}=cr^{\frac{n-sp}{p}}. $$
\end{remark}
\smallskip
The proof of the following theorem is adapted from the proof of \cite[Theorem~6.23]{KiLeVa}.
\begin{proof}[\textbf{\protect{Proof of \cref{Boundary FSPI}}}]
Suppose $\rn\setminus\Omega$ is uniformly $(s,p)$-fat set with a constant $\gamma$ and let $z\in\rn\setminus\Omega,\;r>0.$ Let $u\in C_c^\infty(\Omega)$. Then $\rn\setminus\Omega\subset Z(u):=\{x\in\rn: u(x)=0\}.$ Now, using \cref{cap properties} and by definition of the $(s,p)$ fat set, we obtain
\begin{equation}
    \text{Cap}_{s,p}\left(Z(u)\cap\frac{1}{2}\overline{B_r(z)}, B_r(z)\right)\geq\text{Cap}_{s,p}\left((\rn\setminus\Omega)\cap\frac{1}{2}\overline{B_r(z)}, B_r(z)\right)\geq\gamma\, r^{n-sp}.
\end{equation}
Again, by \cref{FSPI} with the \cref{th 19 for continous function} we have
\begin{multline*}
\left(\fint_{B_{r}(z)}|u(x)|^{q}dx\right)^{\frac{1}{q}}\leq C\left(\frac{ 1}{\text{Cap}_{s,p}\left(Z(u)\cap\frac{1}{2}\overline{ B_r(z)},B_r(z)\right)}\int_{B_r(z)}\int_{B_r(z)}\frac{|u(x)-u(y)|^p}{|x-y|^{n+sp}}dxdy\right)^{\frac{1}{p}}
\\
\leq C\,\gamma^{-1/p}\,r^{s-\frac{n}{p}}\left(\int_{B_r(z)}\int_{B_r(z)}\frac{|u(x)-u(y)|^p}{|x-y|^{n+sp}}dxdy\right)^{\frac{1}{p}}.
\end{multline*}
Therefore, the result follows in this case.\\
   \smallskip
To prove other implication. 
Let $u\in C_c^\infty(\Omega)$, $x\in\Omega$ and $R=\text{dist}(x,\partial\Omega)$. Choose $x_0\in\partial\Omega$ such that $R=|x-x_0|.$ Then by triangle inequality, we have
\begin{equation*}
|u_{B_R(x)}|\leq|u_{B_R(x)}-u_{B_R(x_0)}|+|u_{B_R(x_0)}|. 
\end{equation*}
 Note that, $B_R(x)\cup B_R(x_0)\subset B_{2R}(x)$ and by \cref{FPW q p} with $q=1$, we obtain
\begin{equation}\label{difference of two avg}
\begin{split}
|u_{B_R(x)}-u_{B_R(x_0)}|&\leq |u_{B_R(x)}-u_{B_{2R}(x)}|+|u_{B_{2R}(x)}-u_{B_R(x_0)}|
\\
&\leq C(n)\fint_{B_{2R}(x)}|u(y)-u_{B_{2R}(x)}|dy
\leq C(n,s,p)R^{s-n/p}\,[u]_{s,p,B_{2R}(x)}.
\end{split}
\end{equation}
On the other hand, by H\"older inequality and hypothesis \cref{fat iff FSPI} with a constant $C_1$, we have
\begin{equation}\label{avg x0 est}
\begin{split}
|u_{B_R(x_0)}|\leq\fint_{B_R(x_0)}|u(y)|dy\leq\left(\fint_{B_R(x_0)}|u(y)|^qdy\right)^{1/q}&\leq C_1R^{s-n/p}[u]_{s,p,B_R(x_0)}
\\
&\leq C_1R^{s-n/p}\,[u]_{s,p,B_{2R}(x)}. 
\end{split}
\end{equation}
Combining \cref{difference of two avg} and \cref{avg x0 est}, we have
\begin{equation*}
|u_{B_R(x)}|\leq C(n,s,p,C_1)R^{s-n/p}\,[u]_{s,p,B_{2R}(x)},
\end{equation*}
and this implies
$$
|u_{B_R(x)}|\leq C(n,s,p,C_1)R^s\left(M_{2R}\left(|D^s_{p,\,B_{2R}(x)}u|\right)^p(x)\right)^{1/p},
$$
where 
$$
|D^s_{p,\,B_{2R}(x)}u|(y):=\left(\int_{B_{2R}(x)}\frac{|u(y)-u(z)|^p}{|y-z|^{n+sp}}dz\right)^{1/p}.
$$
By proceeding as in \cref{telescoping lemma}, we obtain
\begin{equation*}
|u(x)-u_{B_R(x)}|\leq C(n,s,p)R^s\left(M_{2R}\left(|D^s_{p,\,B_{2R}(x)}u|\right)^p(x)\right)^{1/p}.
\end{equation*}
Therefore, by above estimates we have for any $x\in\Omega$
\begin{equation}\label{ptwise hardy:th 19}
\begin{split}
|u(x)|&\leq|u(x)-u_{B_R(x)}|+|u_{B_R(x)}|
\\
&\leq C(n,s,p,C_1)R^s\left(M_{2R}\left(|D^s_{p,\,B_{2R}(x)}u|\right)^p(x)\right)^{1/p}
\\
&= C(n,s,p,C_1)\text{dist}(x,\partial\Omega)^s\left(M_{2R}\left(|D^s_{p,\,B_{2\,\text{dist}(x,\partial\Omega)}(x)}u|\right)^p(x)\right)^{1/p}.
\end{split}
\end{equation}
Let $z\in\rn\setminus\Omega,\;r>0.$ To conclude the proof of the theorem,
   it is enough to find a positive constant $C$ such that
   \begin{equation}\label{Th 19 fat proof}
       \int_{B_r(z)}\int_{B_r(z)}\frac{|v(x)-v(y)|^p}{|x-y|^{n+sp}}dxdy\geq C\,r^{n-sp}
   \end{equation}
    whenever $v\in C_c^\infty(B_r(z))$ and $v(x)=1$ for $x\in\left(\rn\setminus\Omega\right)\cap\frac{1}{2}\overline{B_r(z)}.$ Moreover, we may assume that $0\leq v\leq1$, compare to \cref{restriction of function in cap def}. Let $\sigma=1/6.$\\
    \smallskip
\noindent\textbf{Step 1:} If
$$
\fint_{B_{r/2}(z)}v(y)dy>\frac{\sigma^n}{4},
$$
then by \cref{FP q p} with $q=1$ we have
\begin{equation*}
\begin{split}
1<4\sigma^{-n}\fint_{B_{r/2}(z)}v(y)dy
&\leq C(n)\fint_{B_r(z)}|v(y)|dy
\\
&\leq C(n,s,p)\,r^{s-n/p}\left(\int_{B_r(z)}\int_{B_r(z)}\frac{|v(x)-v(y)|^p}{|x-y|^{n+sp}}dxdy\right)^{1/p}
\end{split}
\end{equation*}
and \cref{Th 19 fat proof} follows.\\
\smallskip
\noindent\textbf{Step 2:} If
 $$
\fint_{B_{r/2}(z)}v(y)dy\leq\frac{\sigma^n}{4}.
$$   
Let $F=\left\{y\in B_{\frac{\sigma r}{2}}(z):v(y)<\frac{1}{2}\right\}$. Since $v=1$ in $\left(\rn\setminus\Omega\right)\cap\frac{1}{2}\overline{B_r(z)}$, we have $F\subset\Omega.$ By definition, $v\geq\frac12$ in $B_{\frac{\sigma r}{2}}(z)\setminus F$ and thus
$$
|B_{\frac{\sigma r}{2}}(z)\setminus F|\leq2\int_{B_{\frac{\sigma r}{2}}(z)\setminus F}v(y)dy\leq2\int_{B_{\frac{r}{2}}(z)}v(y)dy\leq\frac{\sigma^n}{2}|B_{r/2}(z)|.
$$
This gives that
\begin{equation}\label{measure of F}
|F|=|B_{\frac{\sigma r}{2}}(z)|-|B_{\frac{\sigma r}{2}}(z)\setminus F|\geq\frac{\sigma^n}{2}|B_{r/2}(z)|.
\end{equation}
Let $\phi\in C_c^\infty(B_{r/2}(z))$ such that $\phi=1$ in $\frac14\overline{B_r(z)}$. Define $u(x)=\phi(x)(1-v(x))$, $x\in\rn.$ Then $u\in C_c^\infty(\Omega)$ and $u=1-v$ in $\frac14\overline{B_r(z)}$. Moreover, $u=0$ in $\Omega^c\cap\frac12\overline{B_r(z)}$ and for $x\in B_{r/4}(z)$, $A\subset B_{r/4}(z)$ we have $|D^s_{p,\,A}u|(x)=|D^s_{p,\,A}v|(x)$. Since, by definition $|u(y)|^p$  is finite for every $y\in F$, there exists a radius $0<R_y\leq2\,\text{dist}(y,\partial\Omega)=:2\,\delta(y)$ such that 
\begin{equation}\label{choice of radius}
 M_{2\delta(y)}\left(|D^s_{p,B_{2\delta(y)}(y)}u|\right)^p(y)\leq2\fint_{B_{R_y}(y)}|D^s_{p,B_{2\delta(y)}(y)}u|^p(x)dx.   
\end{equation}
Observe that the balls $B_{R_y}(y)$ form a cover of F and they have a uniformly bounded radii. By $5R$-covering lemma (see for example \cite[Lemma~1.13]{KiLeVa}) there exist pairwise disjoint balls $B_{R_j}(y_j)$ where $y_j\in F$ and $R_j=R_{y_j}$ are as above, such that $F\subset\bigcup_{j=1}^\infty B_{5R_j}(y_j)$. It follows from \cref{measure of F}
\begin{equation}\label{measure of F comarable}
|B_{r/2}(z)|\leq\frac{2}{\sigma^n}|F|\leq C(n)\sum_{j=1}^\infty |B_{R_j}(y_j)|.
\end{equation}
Let $j\in\mathbb{N}$. Since $y_j\in F\cap B_{\frac{\sigma r}{2}}(z)\subset\Omega\cap B_{\frac{\sigma r}{2}}(z)$ and $z\not\in\Omega$, we have $\delta(y_j):=\text{dist}(y_j,\partial\Omega)<\frac{\sigma r}{2}$. If $x\in B_{R_j}(y_j)$, then 
$$
|x-z|\leq|x-y_j|+|y_j-z|\leq R_j+\frac{\sigma r}{2}\leq2\,\text{dist}(y_j,\partial\Omega)+\frac{\sigma r}{2}<\frac{\sigma r}{2}(2+1)=\frac{r}{4},
$$
and this gives that $B_{R_j}(y_j)\subset B_{r/4}(z).$ Since $y_j\in F$, we have $u(y_j)=1-v(y_j)>\frac12$. By \cref{ptwise hardy:th 19} and the choice of the radius $R_j$ in \cref{choice of radius}, we obtain
\begin{equation*}
\begin{split}
\frac{1}{2^p}\leq|u(y_j)|^p
&\leq C(\delta(y_j))^{sp}M_{2\delta(y_j)}\left(|D^s_{p,B_{2\delta(y_j)}(y_j)}u|\right)^p(y_j)
\\
&\leq C\sigma^{sp}r^{sp}\fint_{B_{R_j}(y_j)}|D^s_{p,B_{2\delta(y_j)}(y_j)}u|^p(x)dx,
\end{split}
\end{equation*}
where $C=C(n,s,p,C_1)>0$ and consequently
$$
|B_{R_j}(y_j)|\leq C\,r^{sp}\int_{B_{R_j}(y_j)}|D^s_{p,B_{2\delta(y_j)}(y_j)}v|^p(x)dx\,\,\text{ for all }j\in\mathbb{N},
$$
where we used the definition of $v.$ Using this into \cref{measure of F comarable}, we get
$$
|B_{r/2}(z)|\leq C\,r^{sp}\sum_{j=1}^\infty\int_{B_{R_j}(y_j)}|D^s_{p,B_{2\delta(y_j)}(y_j)}v|^p(x)dx\leq C\,r^{sp}\int_{B_{r}(z)}|D^s_{p,B_{r}(z)}v|^p(x)dx,
$$
where we also used the fact that the balls $B_{R_j}(y_j)\subset B_{r}(z)$ are pairwise disjoint. This shows that \cref{Th 19 fat proof} holds and the proof is complete.
\end{proof}

\begin{proof}[\textbf{\protect{Proof of \cref{Fractional Hardy Inequality}}}]
Let us choose $q\geq 1$ such that $n<sq<sp$. Now fix $x\in\Omega$ and $x_0\in\partial\Omega$ such that $|x-x_0|=\text{dist}(x,\partial\Omega)=\delta(x)=R$. We denote $\chi=\chi_{B_{4R}(x_0)}$ the characteristic function of $B_{4R}(x_0)$. Let $u\in C_c^\infty(\Omega)$ and consider the natural zero extension of $u$ to $\rn\setminus\Omega$, then by \cref{useful result} we have for almost every $x$, 
\begin{equation}\label{estimate 1}
|u(x)|=|u(x)-u(x_0)|\leq C|x-x_0|^{s-\frac{n}{q}}\left(M_{\frac{n}{q}}(|D^s_p u|\chi)(x)+M_{\frac{n}{q}}(|D^s_p u|\chi)(x_0)\right).
\end{equation}
Using H\"older inequality we have 
\begin{multline*}
    \frac{r^{\frac{n}{q}}}{|B_r(x)|}\int_{B_r(x)}\left(|D^s_p u|\chi\right)(y)dy
    \leq\frac{r^{\frac{n}{q}}}{|B_r(x)|^{1/q}}\left(\int_{B_r(x)}\left(\left(|D^s_p u|\chi\right)(y)\right)^qdy\right)^{1/q}\\
    \leq C(n)\left(\int_{\rn}\left(\left(|D^s_p u|\chi\right)(y)\right)^qdy\right)^{1/q},
\end{multline*} 
and consequently we get
$$M_{\frac{n}{q}}(|D^s_p u|\chi)(x)\leq C(n)\;||(|D^s_p u|\chi)||_{L^q(\rn)}.
$$
Similar computation yields that 
$$
M_{\frac{n}{q}}(|D^s_p u|\chi)(x_0)\leq C(n)\;||(|D^s_p u|\chi)||_{L^q(\rn)}.
$$ 
Combining the above two estimates in \cref{estimate 1} we obtain
\begin{multline*}\label{pointwise frac Hardy}
    |u(x)|
    \leq C |x-x_0|^{s-\frac{n}{q}}\left(\int_{B_{4R}(x_0)}(|D^s_p u|(z))^qdz\right)^{1/q}\\
    \leq C R^{s-\frac{\alpha}{q}}\left(R^{\alpha-n}\int_{B_{5R}(x)}(|D^s_p u|(z))^qdz\right)^{1/q}\\
    \leq C\left(\text{dist}(x,\partial\Omega)\right)^{s-\frac{\alpha}{q}}\left(M_{\alpha,5R}(|D^s_p u|)^q(x)\right)^{1/q}.
\end{multline*}
The above inequality holds for almost every $x\in\rn$. Integrating with respect to the variable $x$ over $\Omega$ with $\alpha=0$, we infer that
\begin{multline*}
  \int_{\Omega}\frac{ |u(x)|^p}{\text{dist}(x,\partial\Omega)^{sp}}dx\leq C\int_{\Omega}\left(M(|D^s_p u|^q(x))\right)^{\frac{p}{q}}dx\leq C\int_{\rn}|D^s_p u|^p(x)dx\\
  =C\int_{\rn}\int_{\rn}\frac{ |u(x)-u(y)|^p}{|x-y|^{n+sp}}dxdy.  
\end{multline*}
In the above estimate we used the Hardy-Littlewood-Wiener maximal function theorem. By density, we conclude the first part of \cref{Fractional Hardy Inequality}. The second part follows from the first part together with \cref{seminorm equivalent}.  This completes the proof of the theorem.
\end{proof}
 
In the following theorem, we prove the validity of pointwise fractional $p$-Hardy inequality \cref{def: ptwise frac HI}. 
\begin{theorem}\label{ptwise frac HI}
   Let $0<s<1$, $1< p<\infty$ such that $1<sp\leq n$, $0\leq\alpha<p$, and let $\Omega$ be an open set in $\rn$ such that it's complement that is $\rn\setminus\Omega$ is uniformly $(s,p)$- fat. Assume that $u\in C^\infty_c(\Omega)$. Then pointwise fractional $p$-Hardy inequality \cref{def: ptwise frac HI} holds true for $\Omega$ that is there exist constants $C=C(n,p,\alpha)>0$ and $\sigma>1$ such that for all $x\in\Omega$
   \begin{equation*}
       |u(x)|\leq C \delta(x)^{s-\frac{\alpha}{p}}\left(M_{\alpha,\;\sigma\delta(x)}|D^s_pu|^p(x)\right)^{\frac{1}{p}}.
   \end{equation*}
\end{theorem}
\begin{proof}
    Let $x\in\Omega.$ Let us choose $x_0\in\partial\Omega$ for which $|x-x_0|=\delta(x)=R.$
    Then, by using the standard telescoping argument as in the proof of \cref{telescoping lemma} we obtain
    $$
    |u(x)-u_{B_R(x_0)}|\leq C R^{s-\frac{\alpha}{p}}\left(M_{\alpha,\;2R}|D^s_p u|^p(x)\right)^{\frac{1}{p}}.
    $$
    Thus, by above estimate we have
    \begin{equation}\label{ptwise value est of function}
        |u(x)|\leq|u(x)-u_{B_R(x_0)}|+|u_{B_R(x_0)}|\leq C R^{s-\frac{\alpha}{p}}\left(M_{\alpha,\;2R}|D^s_p u|^p(x)\right)^{\frac{1}{p}}+|u_{B_R(x_0)}|.
    \end{equation}
Now consider the set $Z(u)=\{x\in\rn:u(x)=0\},$ which is a closed as the function $u\in C^\infty_c(\Omega)$. Then, by using \cref{FSPI} with \cref{th 19 for continous function} and the monotonicity property of the capacity along with hypothesis, we obtain
\begin{multline}\label{avg est}
    |u_{B_R(x_0)}|\leq\left(C\;\text{Cap}_{s,p}\left(Z(u)\cap\frac{1}{2}\overline{B},B\right)^{-1}\int_{B_R(x_0)}|D^s_p u|^p(x)dx\right)^{\frac{1}{p}}\\
    \leq\left(C\;\text{Cap}_{s,p}\left(\Omega^c\cap\frac{1}{2}\overline{B},B\right)^{-1}\int_{B_R(x_0)}|D^s_p u|^p(x)dx\right)^{\frac{1}{p}}\\
    \leq C\left(R^{sp-n}\int_{B_R(x_0)}|D^s_p u|^p(x)dx\right)^{\frac{1}{p}}
    \leq C R^{s-\frac{\alpha}{p}}\left(M_{\alpha,R}|D^s_p u|^p(x)\right)^{\frac{1}{p}}.
\end{multline}
Plugging the estimate \cref{avg est} into \cref{ptwise value est of function} we obtain
$$
|u(x)|\leq C R^{s-\frac{\alpha}{p}}\left(M_{\alpha,\;\sigma R}|D^s_p u|^p(x)\right)^{\frac{1}{p}}.
$$
Since we have chosen $x\in\Omega$ arbitrarily and hence this completes the proof of the theorem.
\end{proof}

As an application of \cref{Boundary FSPI} and \cref{ptwise frac HI} we discuss some examples of domains that are uniformly $(s,p)$-fat set.

\begin{example}
     For $s\in(0,1),\;1<p<\infty$. Then all nonempty closed sets in $\rn$ are uniformly $(s,p)$-fat set provided $sp>n.$
\end{example}
\begin{proof}
    Suppose $E\subset\rn$ be an nonempty closed set. Let $x\in E$ and $r>0$. Consider $u\in C^\infty_c(2B_r(x))$ such that $u\geq1$ in $E\cap\overline{B_r(x)}$. Choose a cutoff function $\rho\in C^\infty_c(3B_r(x))$ such that $0\leq\rho\leq1$, $\rho=1$ in $\overline{B_r(x)}$ and $|\nabla\rho|\leq c/r$. Then, by fractional Morrey's inequality (see, \cite[Corollary~2.7]{BrGoVa}) we have
    \begin{equation}\label{morrey inq}
        |\left(\rho u\right)(y)-\left(\rho u\right)(z)|\leq C|y-z|^{s-\frac{n}{p}}[\rho u]_{s,p,\rn}.
    \end{equation}
    Now, let $y\in E\cap\overline{B_r(x)}$ and $z\in B_{3r}(x)\setminus B_{2r}(x)$. Then, we have $(\rho u)(y)\geq1$ and $(\rho u)(z)=0.$ Thus by \cref{morrey inq}, we obtain
    \begin{equation}\label{lower bdd by morrey}
        [\rho u]_{s,p,\rn}^p\geq C|y-z|^{n-sp}|\left(\rho u\right)(y)-\left(\rho u\right)(z)|\geq C r^{n-sp}.
    \end{equation}
    Therefore, using \cref{seminorm equivalent} and  \cref{lower bdd by morrey} we have
    \begin{equation}\label{lower bdd by morrey1}
        [\rho u]_{s,p,B_{2r}(x)}^p\geq Cr^{n-sp}.
    \end{equation}
    Also, note that
    \begin{multline}\label{seminorm est by rfpi}
        [\rho u]_{s,p,B_{2r}(x)}^p=\int_{B_{2r}(x)}\int_{B_{2r}(x)}\frac{|\left(\rho u\right)(y)-\left(\rho u\right)(z)|^p}{|y-z|^{n+sp}}dydz\\
        =\int_{B_{2r}(x)}\int_{B_{2r}(x)}\frac{|\rho(y)\left(u(y)-u(z)\right)+u(z)\left(\rho(y)-\rho(z)\right)|^p}{|y-z|^{n+sp}}dydz\\
        \leq 2^{p-1}\left(\int_{B_{2r}(x)}\int_{B_{2r}(x)}\frac{| u(y)-u(z)|^p}{|y-z|^{n+sp}}dydz+\frac{C}{r^p}\int_{B_{2r}(x)}\int_{B_{4r}(z)}\frac{| u(z)|^p}{|y-z|^{n+sp-p}}dydz\right)
        \\
        \leq 2^{p-1}\left(\int_{B_{2r}(x)}\int_{B_{2r}(x)}\frac{| u(y)-u(z)|^p}{|y-z|^{n+sp}}dydz+\frac{C}{r^{sp}}\int_{B_{2r}(x)}|u(z)|^pdz\right)
        \\
        \leq C\int_{B_{2r}(x)}\int_{B_{2r}(x)}\frac{| u(y)-u(z)|^p}{|y-z|^{n+sp}}dydz.
    \end{multline}
    In the last estimate we have used the fractional Poincar\'e inequality (see, \cite{BrLiPa}) and then \cref{seminorm equivalent}. Now, combining the estimates \cref{lower bdd by morrey1}, \cref{seminorm est by rfpi} we obtain 
    $$
    \int_{B_{2r}(x)}\int_{B_{2r}(x)}\frac{| u(y)-u(z)|^p}{|y-z|^{n+sp}}dydz\geq C r^{n-sp}.
    $$
    Since $u$ is arbitrary and thus taking infimum over all such functions to get the desired result.
\end{proof}

\begin{example}
Let $0<s<1$ and $1\leq p<\infty$. Suppose $E\subset\rn$ be a closed set such that it satisfies the measure density condition 
\begin{equation}\label{mdc}
    |E\cap B_r(x)|\geq c\;|B_r(x)|
\end{equation}
for all $x\in E$ and radii $r>0$, and for some constant $c>0$. Then $E$ is uniformly $(s,p)$-fat set.
\begin{proof}
    Let $x\in E$ and $r>0$. Let $u\in C_c^\infty(B_{2r}(x))$ such that  $u\geq1$ in $E\cap\overline{B_r(x)}$. Then, by hypothesis \cref{mdc} and using \cref{FP q p} for $q=p$ we obtain
    \begin{multline*}
        c\;|B_r(x)|\leq|E\cap B_r(x)|\leq\int_{E\cap B_r(x)}|u(y)|^pdy\leq\int_{B_{2r}(x)}|u(y)|^pdy
        \\
        \leq C(n,s,p)\,r^{sp}\int_{B_{2r}(x)}\int_{B_{2r}(x)}\frac{|u(y)-u(z)|^p}{|y-z|^{n+sp}}dydz.
    \end{multline*}
    Thus, 
    $$
    \int_{B_{2r}(x)}\int_{B_{2r}(x)}\frac{|u(y)-u(z)|^p}{|y-z|^{n+sp}}dydz\geq c\times C(n,s,p)\,r^{n-sp}.
    $$
    This proves that $E$ is uniformly $(s,p)$-fat set with a constant $\gamma=c\times C(n,s,p).$
\end{proof}
\end{example}

\textbf{Acknowledgements:} The author wish to thank Lorenzo Brasco and Juha Kinnunen for some useful discussions on the content of the paper. The author would like to show his gratitude to the anonymous referee for his/her insightful suggestions and comments that substantially improved the presentation of the paper. This work is partially supported by the SERB WEA grant no. WEA/2020/000005.
    

\bibliography{loc}
\bibliographystyle{plain}

\end{document}